\documentclass[twoside,reqno]{amsart}
\usepackage{amsfonts}
\usepackage{graphicx}
\usepackage{amscd}
\usepackage{hyperref}
\usepackage{cite}

\newtheorem{theorem}{Theorem}
\theoremstyle{plain}

\newtheorem{corollary}{Corollary}

\newtheorem{lemma}{Lemma}

\newtheorem{remark}{Remark}

\numberwithin{equation}{section}

\input{tcilatex}

\begin{document}
\title[Hierarchical fixed point problems and variational inequalities]{%
Strong convergence with a modified iterative projection method for
hierarchical fixed point problems and variational inequalities}
\author{\.{I}brahim Karahan$^{\ast }$}
\address{$^{\ast }$(Corresponding Author) Department of Mathematics, Faculty
of Science, Erzurum Technical University, Erzurum, 25240, Turkey}
\email{ibrahimkarahan@erzurum.edu.tr}
\author{Murat \"{O}zdemir}
\address{Department of Mathematics, Faculty of Science, Ataturk University,
Erzurum, 25240, Turkey.}
\email{mozdemir@atauni.edu.tr}

\begin{abstract}
Let $C$ be a nonempty closed convex subset of a real Hilbert space $H$. Let $%
\left\{ T_{n}\right\} :C\rightarrow H$ be a sequence of nearly nonexpansive
mappings such that $\tciFourier :=\tbigcap_{i=1}^{\infty }F\left(
T_{i}\right) \neq \emptyset $. Let $V:C\rightarrow H$ be a $\gamma $%
-Lipschitzian mapping and $F:C\rightarrow H$ be a $L$-Lipschitzian and $\eta 
$-strongly monotone operator. This paper deals with a modified iterative
projection method for approximating a solution of the hierarchical fixed
point problem. It is shown that under certain approximate assumptions on the
operators and parameters, the modified iterative sequence $\{x_{n}\}$
converges strongly to $x^{\ast }\in \tciFourier $ which is also the unique
solution of the following variational inequality: 
\begin{equation*}
\left\langle \left( \rho V-\mu F\right) x^{\ast },x-x^{\ast }\right\rangle
\leq 0,\text{ }\forall x\in \tciFourier .
\end{equation*}%
As a special case, this projection method can be used to find the minimum
norm solution of above variational inequality; namely, the unique solution $%
x^{\ast }$ to the quadratic minimization problem: $x^{\ast }=\func{argmin}%
_{x\in \tciFourier }\left\Vert x\right\Vert ^{2}$. The results here improve
and extend some recent corresponding results of other authors.
\end{abstract}

\maketitle

\section{Introduction}

Throughout this paper, we assume that $H$ is a real Hilbert space whose
inner product and norm are denoted by $\left\langle \cdot ,\cdot
\right\rangle $ and $\left\Vert \cdot \right\Vert $, respectively, and $C$
is a nonempty closed convex subset of $H$. The set of fixed points of a
mapping $T$ is denoted by $Fix(T)$, that is, $Fix(T)=\{x\in H:Tx=x\}$. Below
we gather some basic definitions and results which are needed in the
subsequent sections. Recall that a mapping $T:C\rightarrow H$ is called $L$%
-Lipschitzian if there exits a constant $L>0$ such that $\left\Vert
Tx-Ty\right\Vert \leq L\left\Vert x-y\right\Vert $, $\forall x,y\in C$. In
particular, if $L\in \lbrack 0,1)$, then $T$ is said to be a contraction; if 
$L=1$, then $T$ is called a nonexpansive mapping. $T$ is called nearly
nonexpansive \cite{AOS1,AOS2} with respect to a fixed sequence $\{a_{n}\}$
in $[0,\infty )$ with $a_{n}\rightarrow 0$ if $\left\Vert
T^{n}x-T^{n}y\right\Vert \leq \left\Vert x-y\right\Vert +a_{n},$ $\forall
x,y\in C$ and $n\geq 1$.

A mapping $F:C\rightarrow H$ is called $\eta $-strongly monotone if there
exists a constant $\eta \geq 0$ such that%
\begin{equation*}
\left\langle Fx-Fy,x-y\right\rangle \geq \eta \left\Vert x-y\right\Vert ^{2},%
\text{ }\forall x,y\in C.
\end{equation*}%
In particular, if $\eta =0$, then $F$ is said to be monotone.

It is well known that for any $x\in H,$ there exists a unique point $%
y_{0}\in C$ such that%
\begin{equation*}
\left\Vert x-y_{0}\right\Vert =\inf \left\{ \left\Vert x-y\right\Vert :y\in
C\right\} ,
\end{equation*}%
where $C$ is a nonempty closed convex subset of $H$. We denote $y_{0}$ by $%
P_{C}x,$ where $P_{C}$ is called the metric projection of $H$ onto $C.$ It
is easy to see $P_{C}$ is a nonexpansive mapping.

Let $S:C\rightarrow H$ be a nonexpansive mapping. The following problem is
called a hierarchical fixed point problem: Find $x^{\ast }\in Fix(T)$ such
that%
\begin{equation}
\left\langle x^{\ast }-Sx^{\ast },x-x^{\ast }\right\rangle \geq 0\text{, \ }%
x\in Fix(T).  \label{Hyr}
\end{equation}%
The problem ($\ref{Hyr}$) is equivalent to the following fixed point
problem: to find an $x^{\ast }\in C$ that satisfies $x^{\ast
}=P_{Fix(T)}Sx^{\ast }$. We know that $Fix(T)$ is closed and convex, so the
metric projection $P_{Fix(T)}$ is well defined.

It is known that the hierarchical fixed point problem ($\ref{Hyr}$) links
with some monotone variational inequalities and convex programming problems;
see \cite{CMMY,T1,YCL,GVC,YC,TH}. Various methods have been proposed to
solve the hierarchical fixed point problem; see Moudafi in \cite{M}, Mainge
and Moudafi in \cite{MM}, Yao and Liou in \cite{YY}, Xu in \cite{X}, Marino
and Xu in \cite{MX2}\ and Bnouhachem and Noor in \cite{BN}.

In 2006, Marino and Xu \cite{MX1} introduced the viscosity iterative method
for nonexpansive mappings. They considered the following general iterative
method:%
\begin{equation}
x_{n+1}=\alpha _{n}\gamma f(x_{n})+\left( 1-\alpha _{n}A\right) Tx_{n},\text{
}\forall n\geq 0,  \label{f1}
\end{equation}%
where $f$ is a contraction, $T$ is a nonexpansive mapping and $A$ is a
strongly positive bounded linear operator on $H$; that is, there is a
constant $\gamma >0$ such that $\left\langle Ax,x\right\rangle \geq \gamma
\left\Vert x\right\Vert ,$ $\forall x\in H$. They proved that the sequence $%
\{x_{n}\}$ generated by (\ref{f1}) converges strongly to the unique solution
of the variational inequality 
\begin{equation}
\left\langle \left( \gamma f-A\right) x^{\ast },x-x^{\ast }\right\rangle
\leq 0,\text{ }\forall x\in C,  \label{v1}
\end{equation}%
which is the optimality condition for the minimization problem%
\begin{equation*}
\min_{x\in C}\frac{1}{2}\left\langle Ax,x\right\rangle -h(x)
\end{equation*}%
where $h$ is a potential function for $\gamma f$, i.e., $h^{\prime
}(x)=\gamma f(x)$ for all $x\in H$.

On the other hand, in 2010, Tian \cite{T1} proposed an implicit and an
explicit schemes on combining the iterative methods of Yamada \cite{Yamu}
and Marino and Xu \cite{MX1}. He also proved the strong convergence of these
two schemes to a fixed point of a nonexpansive mapping $T$ defined on a real
Hilbert space under suitable conditions. In the same year, Ceng et al. \cite%
{Ceng} investigated the following iterative method:%
\begin{equation}
x_{n+1}=P_{C}\left[ \alpha _{n}\rho Vx_{n}+\left( 1-\alpha _{n}\mu F\right)
Tx_{n}\right] ,\text{ }\forall n\geq 0,  \label{f2}
\end{equation}%
where $F$ is a $L$-Lipschitzian and $\eta $-strongly monotone operator with
constants $L,\eta >0$ and$\ V$ is a $\gamma $-Lipschitzian (possibly
non-self) mapping with constant $\gamma \geq 0$ such that $0<\mu <\frac{%
2\eta }{L^{2}}$ and $0\leq \rho \gamma <1-\sqrt{1-\mu \left( 2\eta -\mu
L^{2}\right) }$. They proved that under some approximate assumptions on the
operators and parameters, the sequence $\{x_{n}\}$ generated by (\ref{f2})
converges strongly to the unique solution of the variational inequality%
\begin{equation}
\left\langle \left( \rho V-\mu F\right) x^{\ast },x-x^{\ast }\right\rangle
\leq 0,\text{ }\forall x\in Fix(T).  \label{v2}
\end{equation}

Fix a sequence $\{a_{n}\}$ in $[0,\infty )$ with $a_{n}\rightarrow 0$ and
let $\left\{ T_{n}\right\} $ be a sequence of mappings from $C$ into $H$.
Then, the sequence $\left\{ T_{n}\right\} $ is called a sequence of nearly
nonexpansive mappings \cite{WSY,SKS} with respect to a sequence $\{a_{n}\}$
if%
\begin{equation}
\left\Vert T_{n}x-T_{n}y\right\Vert \leq \left\Vert x-y\right\Vert +a_{n}%
\text{, }\forall x,y\in C\text{, }\forall n\geq 1\text{.}  \label{nearly}
\end{equation}%
It is obvious that the sequence of nearly nonexpansive mappings is a wider
class of sequence of nonexpansive mappings. Recently, in 2012, Sahu et al. 
\cite{SKS} introduced the following iterative process for the sequence of
nearly nonexpansive mappings $\left\{ T_{n}\right\} $ defined by (\ref%
{nearly})%
\begin{equation}
x_{n+1}=P_{C}\left[ \alpha _{n}\rho Vx_{n}+\left( 1-\alpha _{n}\mu F\right)
T_{n}x_{n}\right] ,\text{ }\forall n\geq 1.  \label{f3}
\end{equation}%
They proved that the sequence $\{x_{n}\}$ generated by (\ref{f3}) converges
strongly to the unique solution of the variational inequality (\ref{v2}).

Very recently, in 2013, Wang and Xu \cite{WX} investigated an iterative
method for a hierarchical fixed point problem by%
\begin{equation}
\left\{ 
\begin{array}{c}
y_{n}=\beta _{n}Sx_{n}+\left( 1-\beta _{n}\right) x_{n},\text{ \ \ \ \ \ \ \
\ \ \ \ \ \ \ \ \ \ \ \ \ \ \ } \\ 
x_{n+1}=P_{C}\left[ \alpha _{n}\rho Vx_{n}+\left( I-\alpha _{n}\mu F\right)
Ty_{n}\right] ,\text{ }\forall n\geq 0%
\end{array}%
\right.  \label{f4}
\end{equation}%
where $S:C\rightarrow C$ is a nonexpansive mapping. They proved that under
some approximate assumptions on the operators and parameters, the sequence $%
\{x_{n}\}$ generated by (\ref{f4}) converges strongly to the unique solution
of the variational inequality (\ref{v2}).

In this paper, motivated by the work of Wang and Xu \cite{WX} and Sahu et
al. \cite{SKS} and by the recent work going in this direction, we introduce
an modified iterative projection method and prove a strong convergence
theorem based on this method for computing an element of the set of common
fixed points of a sequence $\left\{ T_{n}\right\} $ of nearly nonexpansive
mappings defined by (\ref{nearly}) which is also an unique\ solution of the
variational inequality (\ref{v2}). The presented method improves and
generalizes many known results for solving variational inequality problems
and hierarchical fixed point problems, see, e.g., \cite{T1,MX1,Ceng,SKS,WX}
and relevant references cited therein.

\section{Preliminaries}

Let $\left\{ x_{n}\right\} $ be a sequence in a Hilbert space $H$ and $x\in
H $. Throughout this paper, $x_{n}\rightarrow x$ denotes that $\left\{
x_{n}\right\} $ strongly converges to $x$ and $x_{n}\rightharpoonup x$
denotes that $\left\{ x_{n}\right\} $ weakly converges to $x$.

Let $C$ be a nonempty subset of a real Hilbert space $H$ and $%
T_{1},T_{2}:C\rightarrow H$ be two mappings. We denote $\mathcal{B}\left(
C\right) $, the collection of all bounded subsets of $C$. The deviation
between $T_{1}$ and $T_{2}$ on $B\in $ $\mathcal{B}\left( C\right) $,
denoted by $\mathfrak{D}_{B}\left( T_{1},T_{2}\right) ,$ is defined by%
\begin{equation*}
\mathfrak{D}_{B}\left( T_{1},T_{2}\right) =\sup \left\{ \left\Vert
T_{1}x-T_{2}x\right\Vert :x\in B\right\} .
\end{equation*}

The following lemmas will be used in the next section.

\begin{lemma}
\cite{WSY}\label{Wang} Let $C$ be a nonempty closed bounded subset of a
Banach space $X$ and $\{T_{n}\}$ be a sequence of nearly nonexpansive
self-mappings on $C$ with a sequence $\{a_{n}\}$ such that $\mathfrak{D}%
_{C}\left( T_{n},T_{n+1}\right) <\infty $. Then, for each $x\in C$, $%
\{T_{n}x\}$ converges strongly to some point of $C$. Moreover, if $T$ is a
mapping from $C$ into itself defined by $Tz=\lim_{n\rightarrow \infty
}T_{n}z $ for all $z\in C$, then $T$ is nonexpansive and $\lim_{n\rightarrow
\infty }\mathfrak{D}_{C}\left( T_{n},T\right) =0$.
\end{lemma}

\begin{lemma}
\cite{Ceng}\label{str} Let $V:C\rightarrow H$ be a $\gamma $-Lipschitzian
mapping with a constant $\gamma \geq 0$ and let $F:C\rightarrow H$ be a $L$%
-Lipschitzian and $\eta $-strongly monotone operator with constants $L,\eta
>0$. Then for $0\leq \rho \gamma <\mu \eta ,$%
\begin{equation*}
\left\langle \left( \mu F-\rho V\right) x-\left( \mu F-\rho V\right)
y,x-y\right\rangle \geq \left( \mu \eta -\rho \gamma \right) \left\Vert
x-y\right\Vert ^{2},\text{ }\forall x,y\in C.
\end{equation*}%
That is, $\mu F-\rho V$ is strongly monotone with coefficient $\mu \eta
-\rho \gamma $.
\end{lemma}

\begin{lemma}
\cite{Yamu}\label{cont} Let $C$ be a nonempty subset of a real Hilbert space 
$H.$ Suppose that $\lambda \in \left( 0,1\right) $ and $\mu >0$. Let $%
F:C\rightarrow H$ be a $L$-Lipschitzian and $\eta $-strongly monotone
operator on $C$. Define the mapping $G:C\rightarrow H$ by%
\begin{equation*}
Gx=x-\lambda \mu Fx\text{, }\forall x\in C.
\end{equation*}%
Then, $G$ is a contraction that provided $\mu <\frac{2\eta }{L^{2}}$. More
precisely, for $\mu \in \left( 0,\frac{2\eta }{L^{2}}\right) ,$%
\begin{equation*}
\left\Vert Gx-Gy\right\Vert \leq \left( 1-\lambda \nu \right) \left\Vert
x-y\right\Vert \text{, }\forall x,y\in C,
\end{equation*}%
where $\nu =1-\sqrt{1-\mu \left( 2\eta -\mu L^{2}\right) }$.
\end{lemma}

\begin{lemma}
\label{b}\cite{Kirk} Let $C$ be a nonempty closed convex subset of a real
Hilbert space $H,$ and $T$ be a nonexpansive self-mapping on $C.$ If $%
Fix\left( T\right) \neq \emptyset ,$ then $I-T$ is demiclosed; that is
whenever $\left\{ x_{n}\right\} $ is a sequence in $C$ weakly converging to
some $x\in C$ and the sequence $\left\{ \left( I-T\right) x_{n}\right\} $
strongly converges to some $y$, it follows that $\left( I-T\right) x=y.$
Here $I$ is the identity operator of $H.$
\end{lemma}

\begin{lemma}
\label{Y}\cite{xu} Assume that $\left\{ x_{n}\right\} $ is a sequence of
nonnegative real numbers satisfying the conditions%
\begin{equation*}
x_{n+1}\leq \left( 1-\alpha _{n}\right) x_{n}+\alpha _{n}\beta _{n},\text{ }%
\forall n\geq 1
\end{equation*}%
where $\left\{ \alpha _{n}\right\} $ \ and $\left\{ \beta _{n}\right\} $ are
sequences of real numbers such that 
\begin{eqnarray*}
&\text{(i)}&\left\{ \alpha _{n}\right\} \subset \left[ 0,1\right] \text{ and 
}\tsum_{n=1}^{\infty }\alpha _{n}=\infty \text{,\ \ \ \ \ \ \ \ \ \ \ \ \ \
\ \ \ \ \ \ \ \ \ \ \ \ \ \ \ \ \ \ \ \ \ \ \ \ \ \ \ \ \ \ \ \ \ \ \ \ \ }
\\
&\text{(ii)}&\limsup_{n\rightarrow \infty }\beta _{n}\leq 0\text{.}
\end{eqnarray*}%
Then $\lim_{n\rightarrow \infty }x_{n}=0.$
\end{lemma}

\section{Main Results}

Now, we give the main results in this paper.

\begin{theorem}
\label{X} Let $C$ be a nonempty closed convex subset of a real Hilbert space 
$H.$ Let $S:C\rightarrow H$ be a nonexpansive mapping and $\left\{
T_{n}\right\} $ be a sequence of nearly nonexpansive mappings with the
sequence $\left\{ a_{n}\right\} $ such that $\tciFourier
:=\tbigcap_{n=1}^{\infty }Fix\left( T_{n}\right) \neq \emptyset $. Suppose
that $Tx=\lim_{n\rightarrow \infty }T_{n}x$ for all $x\in C$ and $Fix\left(
T\right) =\tciFourier .$ Let $V:C\rightarrow H$ be a $\gamma $-Lipschitzian
mapping, $F:C\rightarrow H$ be a $L$-Lipschitzian and $\eta $-strongly
monotone operator such that these coefficients satisfy $0<\mu <\frac{2\eta }{%
L^{2}}$, $0\leq \rho \gamma <\nu $, where $\nu =1-\sqrt{1-\mu \left( 2\eta
-\mu L^{2}\right) }$. For an arbitrarily initial value $x_{1},$ consider the
sequence $\left\{ x_{n}\right\} $ in $C$ generated by%
\begin{equation}
\left\{ 
\begin{array}{c}
y_{n}=P_{C}\left[ \beta _{n}Sx_{n}+\left( 1-\beta _{n}\right) x_{n}\right] ,%
\text{ \ \ \ \ \ \ \ \ \ \ \ \ \ \ \ \ } \\ 
x_{n+1}=P_{C}\left[ \alpha _{n}\rho Vx_{n}+\left( I-\alpha _{n}\mu F\right)
T_{n}y_{n}\right] ,\text{ }n\geq 1,%
\end{array}%
\right.  \label{4}
\end{equation}%
where $\left\{ \alpha _{n}\right\} $ and $\left\{ \beta _{n}\right\} $ are
sequences in $\left[ 0,1\right] $ satisfying the conditions:%
\begin{eqnarray*}
&(C1)&\lim_{n\rightarrow \infty }\alpha _{n}=0\text{, }\tsum_{n=1}^{\infty
}\alpha _{n}=\infty \text{ and }\lim_{n\rightarrow \infty }\beta _{n}=0\text{%
;} \\
&(C2)&\lim_{n\rightarrow \infty }\frac{a_{n}}{\alpha _{n}}=0\text{, }%
\lim_{n\rightarrow \infty }\frac{\beta _{n}}{\alpha _{n}}=0\text{, }%
\lim_{n\rightarrow \infty }\frac{\left\vert \alpha _{n}-\alpha
_{n-1}\right\vert }{\alpha _{n}}=0\text{ and } \\
&&\lim_{n\rightarrow \infty }\frac{\left\vert \beta _{n}-\beta
_{n-1}\right\vert }{\alpha _{n}}=0\text{;} \\
&(C3)&\lim_{n\rightarrow \infty }\mathfrak{D}_{B}\left( T_{n},T_{n+1}\right)
=0\text{ and }\lim_{n\rightarrow \infty }\frac{\mathfrak{D}_{B}\left(
T_{n},T_{n+1}\right) }{\alpha _{n}}=0\text{ for each }B\in \mathcal{B}\left(
C\right) \text{. \ \ \ \ \ }
\end{eqnarray*}%
Then, the sequence $\left\{ x_{n}\right\} $ converges strongly to $x^{\ast
}\in \tciFourier $, where $x^{\ast }$\ is the unique solution of the
variational inequality%
\begin{equation}
\left\langle \left( \rho V-\mu F\right) x^{\ast },x-x^{\ast }\right\rangle
\leq 0,\text{ }\forall x\in \tciFourier .  \label{1}
\end{equation}%
In particular, the point $x^{\ast }$\ is the minimum norm fixed point of $T,$
that is $x^{\ast }$\ is the unique solution of the quadratic minimization
problem 
\begin{equation*}
x^{\ast }=\func{argmin}_{x\in \tciFourier }\left\Vert x\right\Vert ^{2}.
\end{equation*}
\end{theorem}

\begin{proof}
Since the mapping $T$ is defined by $Tx=\lim_{n\rightarrow \infty }T_{n}x$
for all $x\in C,$ by Lemma \ref{Wang}, $T$ is a nonexpansive mapping, and $%
Fix\left( T\right) \neq \emptyset $. Moreover, since the operator $\mu
F-\rho V$ is $\left( \mu \eta -\rho \gamma \right) $-strongly monotone by
Lemma \ref{str}, we get the uniqueness of the solution of the variational
inequality (\ref{1}). Let denote this solution by $x^{\ast }\in Fix\left(
T\right) =\tciFourier $.

Now, we divide our proof into six steps.

\textbf{Step 1. }First we show that the sequences $\left\{ x_{n}\right\} $
is bounded. From hypothesis (C2), since $\lim_{n\rightarrow \infty }\frac{%
\beta _{n}}{\alpha _{n}}=0$, without loss of generality, we may assume that $%
\beta _{n}\leq \alpha _{n}$, for $n\geq 1$. Let $p\in \tciFourier $ and $%
t_{n}=\alpha _{n}\rho Vx_{n}+\left( I-\alpha _{n}\mu F\right) T_{n}y_{n}$.\
Then we have%
\begin{eqnarray}
\left\Vert y_{n}-p\right\Vert &=&\left\Vert P_{C}\left[ \beta
_{n}Sx_{n}+\left( 1-\beta _{n}\right) x_{n}\right] -P_{C}p\right\Vert  \notag
\\
&\leq &\left\Vert \beta _{n}Sx_{n}+\left( 1-\beta _{n}\right)
x_{n}-p\right\Vert  \notag \\
&\leq &\left( 1-\beta _{n}\right) \left\Vert x_{n}-p\right\Vert +\beta
_{n}\left\Vert Sx_{n}-p\right\Vert  \notag \\
&\leq &\left( 1-\beta _{n}\right) \left\Vert x_{n}-p\right\Vert +\beta
_{n}\left\Vert Sx_{n}-Sp\right\Vert +\beta _{n}\left\Vert Sp-p\right\Vert 
\notag \\
&\leq &\left\Vert x_{n}-p\right\Vert +\beta _{n}\left\Vert Sp-p\right\Vert ,
\label{2}
\end{eqnarray}%
and%
\begin{eqnarray}
\left\Vert x_{n+1}-p\right\Vert &=&\left\Vert P_{C}t_{n}-P_{C}p\right\Vert 
\notag \\
&\leq &\left\Vert t_{n}-p\right\Vert  \notag \\
&=&\left\Vert \alpha _{n}\rho Vx_{n}+\left( I-\alpha _{n}\mu F\right)
T_{n}y_{n}-p\right\Vert  \notag \\
&\leq &\alpha _{n}\left\Vert \rho Vx_{n}-\mu Fp\right\Vert +\left\Vert
\left( I-\alpha _{n}\mu F\right) T_{n}y_{n}-\left( I-\alpha _{n}\mu F\right)
T_{n}p\right\Vert  \notag \\
&\leq &\alpha _{n}\rho \gamma \left\Vert x_{n}-p\right\Vert +\alpha
_{n}\left\Vert \rho Vp-\mu Fp\right\Vert  \notag \\
&&+\left( 1-\alpha _{n}\nu \right) \left( \left\Vert y_{n}-p\right\Vert
+a_{n}\right) .  \label{3}
\end{eqnarray}%
From (\ref{2}) and (\ref{3}), we get%
\begin{eqnarray}
\left\Vert x_{n+1}-p\right\Vert &\leq &\alpha _{n}\rho \gamma \left\Vert
x_{n}-p\right\Vert +\alpha _{n}\left\Vert \rho Vp-\mu Fp\right\Vert  \notag
\\
&&+\left( 1-\alpha _{n}\nu \right) \left( \left\Vert x_{n}-p\right\Vert
+\beta _{n}\left\Vert Sp-p\right\Vert +a_{n}\right)  \notag \\
&\leq &\left( 1-\alpha _{n}\left( \nu -\rho \gamma \right) \right)
\left\Vert x_{n}-p\right\Vert +\alpha _{n}\left( \left\Vert \rho Vp-\mu
Fp\right\Vert +\left\Vert Sp-p\right\Vert +a_{n}\right)  \notag \\
&\leq &\left( 1-\alpha _{n}\left( \nu -\rho \gamma \right) \right)
\left\Vert x_{n}-p\right\Vert  \notag \\
&&+\alpha _{n}\left( \nu -\rho \gamma \right) \left[ \frac{1}{\left( \nu
-\rho \gamma \right) }\left( \left\Vert \rho Vp-\mu Fp\right\Vert
+\left\Vert Sp-p\right\Vert +\frac{a_{n}}{\alpha _{n}}\right) \right] .
\label{e3}
\end{eqnarray}%
Note that $\frac{a_{n}}{\alpha _{n}}\rightarrow 0$ as $n\rightarrow \infty $%
, so there exists a constant $M>0$ such that%
\begin{equation*}
\left\Vert \rho Vp-\mu Fp\right\Vert +\left\Vert Sp-p\right\Vert +\frac{a_{n}%
}{\alpha _{n}}\leq M\text{, }\forall n\geq 1\text{.}
\end{equation*}%
Thus, from (\ref{e3}) we have%
\begin{equation*}
\left\Vert x_{n+1}-p\right\Vert \leq \left( 1-\alpha _{n}\left( \nu -\rho
\gamma \right) \right) \left\Vert x_{n}-p\right\Vert +\alpha _{n}\left( \nu
-\rho \gamma \right) \frac{M}{\left( \nu -\rho \gamma \right) }.
\end{equation*}%
By induction, we get%
\begin{equation*}
\left\Vert x_{n+1}-p\right\Vert \leq \max \left\{ \left\Vert
x_{1}-p\right\Vert ,\frac{M}{\left( \nu -\rho \gamma \right) }\right\} .
\end{equation*}%
Hence, we obtain that $\left\{ x_{n}\right\} $ is bounded. So, the sequences 
$\left\{ y_{n}\right\} $,$\left\{ Tx_{n}\right\} $,$\left\{ Sx_{n}\right\} $,%
$\left\{ Vx_{n}\right\} $ and $\left\{ FTy_{n}\right\} $ are bounded.

\textbf{Step 2.} Now, we show that $\lim_{n\rightarrow \infty }\left\Vert
x_{n+1}-x_{n}\right\Vert =0$. By using the iteration (\ref{4}), we have%
\begin{eqnarray}
\left\Vert y_{n}-y_{n-1}\right\Vert &=&\left\Vert P_{C}\left[ \beta
_{n}Sx_{n}+\left( 1-\beta _{n}\right) x_{n}\right] -P_{C}\left[ \beta
_{n-1}Sx_{n-1}-\left( 1-\beta _{n-1}\right) x_{n-1}\right] \right\Vert 
\notag \\
&\leq &\beta _{n}\left\Vert Sx_{n}-Sx_{n-1}\right\Vert +\left( 1-\beta
_{n}\right) \left\Vert x_{n}-x_{n-1}\right\Vert  \notag \\
&&+\left\vert \beta _{n}-\beta _{n-1}\right\vert \left( \left\Vert
Sx_{n-1}\right\Vert +\left\Vert x_{n-1}\right\Vert \right)  \notag \\
&\leq &\left\Vert x_{n}-x_{n-1}\right\Vert +\left\vert \beta _{n}-\beta
_{n-1}\right\vert M_{1},  \label{e4}
\end{eqnarray}%
where $M_{1}$ is a constant such that $\sup_{n\geq 1}\left\{ \left\Vert
Sx_{n}\right\Vert +\left\Vert x_{n}\right\Vert \right\} \leq M_{1}.$ Also,
by using the inequality (\ref{e4}), we get%
\begin{eqnarray*}
\left\Vert x_{n+1}-x_{n}\right\Vert &\leq &\left\Vert
P_{C}t_{n}-P_{C}t_{n-1}\right\Vert \\
&\leq &\left\Vert \alpha _{n}\rho Vx_{n}+\left( I-\alpha _{n}\mu F\right)
T_{n}y_{n}\right. \\
&&\left. -\alpha _{n-1}\rho Vx_{n-1}+\left( I-\alpha _{n-1}\mu F\right)
T_{n-1}y_{n-1}\right\Vert \\
&\leq &\left\Vert \alpha _{n}\rho V\left( x_{n}-x_{n-1}\right) +\left(
\alpha _{n}-\alpha _{n-1}\right) \rho Vx_{n-1}+\left( I-\alpha _{n}\mu
F\right) T_{n}y_{n}\right. \\
&&-\left( I-\alpha _{n}\mu F\right) T_{n}y_{n-1}+T_{n}y_{n-1}-T_{n-1}y_{n-1}
\\
&&\left. +\alpha _{n-1}\mu FT_{n-1}y_{n-1}-\alpha _{n}\mu
FT_{n}y_{n-1}\right\Vert \\
&\leq &\alpha _{n}\rho \gamma \left\Vert x_{n}-x_{n-1}\right\Vert +\gamma
\left\vert \alpha _{n}-\alpha _{n-1}\right\vert \left\Vert
Vx_{n-1}\right\Vert \\
&&+\left( 1-\alpha _{n}\nu \right) \left\Vert
T_{n}y_{n}-T_{n}y_{n-1}\right\Vert +\left\Vert
T_{n}y_{n-1}-T_{n-1}y_{n-1}\right\Vert \\
&&+\mu \left\Vert \alpha _{n-1}FT_{n-1}y_{n-1}-\alpha
_{n}FT_{n}y_{n-1}\right\Vert \\
&\leq &\alpha _{n}\rho \gamma \left\Vert x_{n}-x_{n-1}\right\Vert +\gamma
\left\vert \alpha _{n}-\alpha _{n-1}\right\vert \left\Vert
Vx_{n-1}\right\Vert \\
&&+\left( 1-\alpha _{n}\nu \right) \left[ \left\Vert
y_{n}-y_{n-1}\right\Vert +a_{n}\right] +\left\Vert
T_{n}y_{n-1}-T_{n-1}y_{n-1}\right\Vert \\
&&+\mu \left\Vert \alpha _{n-1}\left( FT_{n-1}y_{n-1}-FT_{n}y_{n-1}\right)
-\left( \alpha _{n}-\alpha _{n-1}\right) FT_{n}y_{n-1}\right\Vert \\
&\leq &\alpha _{n}\rho \gamma \left\Vert x_{n}-x_{n-1}\right\Vert +\gamma
\left\Vert Vx_{n-1}\right\Vert +\left( 1-\alpha _{n}v\right) \left\Vert
x_{n}-x_{n-1}\right\Vert \\
&&+\left( 1-\alpha _{n}v\right) \left\vert \beta _{n}-\beta
_{n-1}\right\vert M_{1}+\left( 1-\alpha _{n}v\right) a_{n}+\mathfrak{D}%
_{B}\left( T_{n},T_{n-1}\right) \\
&&+\mu \alpha _{n-1}L\mathfrak{D}_{B}\left( T_{n},T_{n-1}\right) +\left\vert
\alpha _{n}-\alpha _{n-1}\right\vert \left\Vert FT_{n}y_{n-1}\right\Vert \\
&\leq &\left( 1-\alpha _{n}\left( v-\rho \gamma \right) \right) \left\Vert
x_{n}-x_{n-1}\right\Vert \\
&&+\left\vert \alpha _{n}-\alpha _{n-1}\right\vert \left( \gamma \left\Vert
Vx_{n-1}\right\Vert +\left\Vert FT_{n}y_{n-1}\right\Vert \right) \\
&&+\left( 1+\mu \alpha _{n-1}L\right) \mathfrak{D}_{B}\left(
T_{n},T_{n-1}\right) +\left\vert \beta _{n}-\beta _{n-1}\right\vert
M_{1}+a_{n} \\
&\leq &\left( 1-\alpha _{n}\left( v-\rho \gamma \right) \right) \left\Vert
x_{n}-x_{n-1}\right\Vert +\alpha _{n}\left( v-\rho \gamma \right) \delta
_{n},
\end{eqnarray*}%
where%
\begin{equation*}
\delta _{n}=\frac{1}{\left( \nu -\rho \gamma \right) }\left[ 
\begin{array}{c}
\left( 1+\mu \alpha _{n-1}L\right) \frac{\mathfrak{D}_{B}\left(
T_{n},T_{n-1}\right) }{\alpha _{n}} \\ 
+\left( \left\vert \frac{\alpha _{n}-\alpha _{n-1}}{\alpha _{n}}\right\vert
+\left\vert \frac{\beta _{n}-\beta _{n-1}}{\alpha _{n}}\right\vert \right)
M_{2}+\frac{a_{n}}{\alpha _{n}}%
\end{array}%
\right] ,
\end{equation*}%
and%
\begin{equation*}
\sup_{n\geq 1}\left\{ \gamma \left\Vert Vx_{n-1}\right\Vert +\left\Vert
FT_{n}y_{n-1}\right\Vert ,\text{ }M_{1}\right\} \leq M_{2}.
\end{equation*}%
Since $\limsup_{n\rightarrow \infty }\delta _{n}\leq 0$, it follows from
Lemma \ref{Y}, conditions (C2) and (C3) that%
\begin{equation}
\left\Vert x_{n+1}-x_{n}\right\Vert \rightarrow 0\text{ as }n\rightarrow
\infty .  \label{e5}
\end{equation}

\textbf{Step 3.} Next, we show that $\lim_{n\rightarrow \infty }\left\Vert
x_{n}-Tx_{n}\right\Vert =0$ as $n\rightarrow \infty .$ Note that%
\begin{eqnarray*}
\left\Vert x_{n}-T_{n}x_{n}\right\Vert &\leq &\left\Vert
x_{n}-x_{n+1}\right\Vert +\left\Vert x_{n+1}-T_{n}x_{n}\right\Vert \\
&\leq &\left\Vert x_{n}-x_{n+1}\right\Vert +\left\Vert
P_{C}t_{n}-P_{C}T_{n}x_{n}\right\Vert \\
&\leq &\left\Vert x_{n}-x_{n+1}\right\Vert +\left\Vert \alpha _{n}\rho
Vx_{n}+\left( I-\alpha _{n}\mu F\right) T_{n}y_{n}-T_{n}x_{n}\right\Vert \\
&\leq &\left\Vert x_{n}-x_{n+1}\right\Vert +\left\Vert \alpha _{n}\left(
\rho Vx_{n}-\mu FT_{n}y_{n}\right) +T_{n}y_{n}-T_{n}x_{n}\right\Vert \\
&\leq &\left\Vert x_{n}-x_{n+1}\right\Vert +\alpha _{n}\left\Vert \rho
Vx_{n}-\mu FT_{n}y_{n}\right\Vert +\left\Vert y_{n}-x_{n}\right\Vert +a_{n}
\\
&\leq &\left\Vert x_{n}-x_{n+1}\right\Vert +\alpha _{n}\left\Vert \rho
Vx_{n}-\mu FT_{n}y_{n}\right\Vert +\beta _{n}\left\Vert
Sx_{n}-x_{n}\right\Vert +a_{n}.
\end{eqnarray*}%
Since $a_{n}\rightarrow 0$, \ by using (\ref{e5}) and condition (C1),\ we
obtain 
\begin{equation*}
\lim_{n\rightarrow \infty }\left\Vert x_{n}-T_{n}x_{n}\right\Vert =0.
\end{equation*}%
Hence, we have 
\begin{eqnarray*}
\left\Vert x_{n}-Tx_{n}\right\Vert &\leq &\left\Vert
x_{n}-T_{n}x_{n}\right\Vert +\left\Vert T_{n}x_{n}-Tx_{n}\right\Vert \\
&\leq &\left\Vert x_{n}-T_{n}x_{n}\right\Vert +\mathfrak{D}_{B}\left(
T_{n},T\right) \rightarrow 0\text{ as }n\rightarrow \infty .
\end{eqnarray*}

\textbf{Step 4.} Next, we show that $\limsup_{n\rightarrow \infty
}\left\langle \left( \rho V-\mu F\right) x^{\ast },x_{n}-x^{\ast
}\right\rangle \leq 0$, where $x^{\ast }$\ is the unique solution of
variational inequality (\ref{1}). Since the sequence $\left\{ x_{n}\right\} $
is bounded, it has a weak convergent subsequence $\left\{ x_{n_{k}}\right\} $
such that%
\begin{equation*}
\limsup_{n\rightarrow \infty }\left\langle \left( \rho V-\mu F\right)
x^{\ast },x_{n}-x^{\ast }\right\rangle =\limsup_{k\rightarrow \infty
}\left\langle \left( \rho V-\mu F\right) x^{\ast },x_{n_{k}}-x^{\ast
}\right\rangle .
\end{equation*}%
Let $x_{n_{k}}\rightharpoonup \widetilde{x}$, as $k\rightarrow \infty $. It
follows from Lemma \ref{b} that $\widetilde{x}\in Fix\left( T\right)
=\tciFourier $. Hence%
\begin{equation*}
\limsup_{n\rightarrow \infty }\left\langle \left( \rho V-\mu F\right)
x^{\ast },x_{n}-x^{\ast }\right\rangle =\left\langle \left( \rho V-\mu
F\right) x^{\ast },\widetilde{x}-x^{\ast }\right\rangle \leq 0.
\end{equation*}

\textbf{Step 5.} Now, we show that the sequence $\left\{ x_{n}\right\} $
converges strongly to $x^{\ast }$ as $n\rightarrow \infty .$ By using the
iteration (\ref{4}), we have%
\begin{eqnarray}
\left\Vert x_{n+1}-x^{\ast }\right\Vert ^{2} &=&\left\langle
P_{C}t_{n}-x^{\ast },x_{n+1}-x^{\ast }\right\rangle  \notag \\
&=&\left\langle P_{C}t_{n}-t_{n},x_{n+1}-x^{\ast }\right\rangle
+\left\langle t_{n}-x^{\ast },x_{n+1}-x^{\ast }\right\rangle .  \label{e6}
\end{eqnarray}%
Since the metric projection $P_{C}$ satisfies the inequality%
\begin{equation*}
\left\langle x-P_{C}x,y-P_{C}x\right\rangle \leq 0,\text{ }\forall x\in
H,y\in C,
\end{equation*}%
and from (\ref{e6}), we get%
\begin{eqnarray*}
\left\Vert x_{n+1}-x^{\ast }\right\Vert ^{2} &\leq &\left\langle
t_{n}-x^{\ast },x_{n+1}-x^{\ast }\right\rangle \\
&=&\left\langle \alpha _{n}\rho Vx_{n}+\left( I-\alpha _{n}\mu F\right)
T_{n}y_{n}-x^{\ast },x_{n+1}-x^{\ast }\right\rangle \\
&=&\left\langle \alpha _{n}\left( \rho Vx_{n}-\mu Fx^{\ast }\right) +\left(
I-\alpha _{n}\mu F\right) T_{n}y_{n}\right. \\
&&\left. -\left( I-\alpha _{n}\mu F\right) T_{n}x^{\ast },x_{n+1}-x^{\ast
}\right\rangle \\
&=&\alpha _{n}\rho \left\langle Vx_{n}-Vx^{\ast },x_{n+1}-x^{\ast
}\right\rangle +\alpha _{n}\left\langle \rho Vx^{\ast }-\mu Fx^{\ast
},x_{n+1}-x^{\ast }\right\rangle \\
&&+\left\langle \left( I-\alpha _{n}\mu F\right) T_{n}y_{n}-\left( I-\alpha
_{n}\mu F\right) T_{n}x^{\ast },x_{n+1}-x^{\ast }\right\rangle .
\end{eqnarray*}%
Hence, from (\ref{2}) and Lemma \ref{cont}, we obtain 
\begin{eqnarray*}
\left\Vert x_{n+1}-x^{\ast }\right\Vert ^{2} &\leq &\alpha _{n}\rho \gamma
\left\Vert x_{n}-x^{\ast }\right\Vert \left\Vert x_{n+1}-x^{\ast
}\right\Vert +\alpha _{n}\left\langle \rho Vx^{\ast }-\mu Fx^{\ast
},x_{n+1}-x^{\ast }\right\rangle \\
&&+\left( 1-\alpha _{n}\nu \right) \left( \left\Vert y_{n}-x^{\ast
}\right\Vert +a_{n}\right) \left\Vert x_{n+1}-x^{\ast }\right\Vert \\
&\leq &\alpha _{n}\rho \gamma \left\Vert x_{n}-x^{\ast }\right\Vert
\left\Vert x_{n+1}-x^{\ast }\right\Vert +\alpha _{n}\left\langle \rho
Vx^{\ast }-\mu Fx^{\ast },x_{n+1}-x^{\ast }\right\rangle \\
&&+\left( 1-\alpha _{n}\nu \right) \left( \left\Vert x_{n}-x^{\ast
}\right\Vert +\beta _{n}\left\Vert Sx^{\ast }-x^{\ast }\right\Vert
+a_{n}\right) \left\Vert x_{n+1}-x^{\ast }\right\Vert \\
&=&\left( 1-\alpha _{n}\left( v-\rho \gamma \right) \right) \left\Vert
x_{n}-x^{\ast }\right\Vert \left\Vert x_{n+1}-x^{\ast }\right\Vert \\
&&+\alpha _{n}\left\langle \rho Vx^{\ast }-\mu Fx^{\ast },x_{n+1}-x^{\ast
}\right\rangle \\
&&+\left( 1-\alpha _{n}v\right) \beta _{n}\left\Vert Sx^{\ast }-x^{\ast
}\right\Vert \left\Vert x_{n+1}-x^{\ast }\right\Vert \\
&&+\left( 1-\alpha _{n}v\right) a_{n}\left\Vert x_{n+1}-x^{\ast }\right\Vert
\\
&\leq &\frac{\left( 1-\alpha _{n}\left( v-\rho \gamma \right) \right) }{2}%
\left( \left\Vert x_{n}-x^{\ast }\right\Vert ^{2}+\left\Vert x_{n+1}-x^{\ast
}\right\Vert ^{2}\right) \\
&&+\alpha _{n}\left\langle \rho Vx^{\ast }-\mu Fx^{\ast },x_{n+1}-x^{\ast
}\right\rangle +\beta _{n}\left\Vert Sx^{\ast }-x^{\ast }\right\Vert
\left\Vert x_{n+1}-x^{\ast }\right\Vert \\
&&+a_{n}\left\Vert x_{n+1}-x^{\ast }\right\Vert ,
\end{eqnarray*}%
which implies that%
\begin{eqnarray*}
\left\Vert x_{n+1}-x^{\ast }\right\Vert ^{2} &\leq &\frac{\left( 1-\alpha
_{n}\left( \nu -\rho \gamma \right) \right) }{\left( 1+\alpha _{n}\left( \nu
-\rho \gamma \right) \right) }\left\Vert x_{n}-x^{\ast }\right\Vert ^{2} \\
&&+\frac{2\alpha _{n}}{\left( 1+\alpha _{n}\left( \nu -\rho \gamma \right)
\right) }\left\langle \rho Vx^{\ast }-\mu Fx^{\ast },x_{n+1}-x^{\ast
}\right\rangle \\
&&+\frac{2\beta _{n}}{\left( 1+\alpha _{n}\left( \nu -\rho \gamma \right)
\right) }\left\Vert Sx^{\ast }-x^{\ast }\right\Vert \left\Vert
x_{n+1}-x^{\ast }\right\Vert \\
&&+\frac{2a_{n}}{\left( 1+\alpha _{n}\left( \nu -\rho \gamma \right) \right) 
}\left\Vert x_{n+1}-x^{\ast }\right\Vert \\
&\leq &\left( 1-\alpha _{n}\left( \nu -\rho \gamma \right) \right)
\left\Vert x_{n}-x^{\ast }\right\Vert ^{2}+\alpha _{n}\left( \nu -\rho
\gamma \right) \theta _{n},
\end{eqnarray*}%
where%
\begin{equation*}
\theta _{n}=\frac{2\alpha _{n}}{\left( 1+\alpha _{n}\left( \nu -\rho \gamma
\right) \right) \left( \nu -\rho \gamma \right) }\left[ 
\begin{array}{c}
\left\langle \rho Vx^{\ast }-\mu Fx^{\ast },x_{n+1}-x^{\ast }\right\rangle +%
\frac{\beta _{n}}{\alpha _{n}}M_{3} \\ 
+\frac{a_{n}}{\alpha _{n}}\left\Vert x_{n+1}-x^{\ast }\right\Vert%
\end{array}%
\right] ,
\end{equation*}%
and 
\begin{equation*}
\sup_{n\geq 1}\left\{ \left\Vert Sx^{\ast }-x^{\ast }\right\Vert \left\Vert
x_{n+1}-x^{\ast }\right\Vert \right\} \leq M_{3}.
\end{equation*}%
Since $\frac{\beta _{n}}{\alpha _{n}}\rightarrow 0$ and $\frac{a_{n}}{\alpha
_{n}}\rightarrow 0,$ we get%
\begin{equation*}
\limsup_{n\rightarrow \infty }\theta _{n}\leq 0.
\end{equation*}%
So, it follows from Lemma \ref{Y} that the sequence $\left\{ x_{n}\right\} $
generated by (\ref{4}) converges strongly to $x^{\ast }\in \tciFourier $
which is the unique solution of variational inequality (\ref{1}).

\textbf{Step 6.} Finally, since the point $x^{\ast }$ is the unique solution
of variational inequality (\ref{1}), in particular if we take $V=0$ and $F=I$
in the variational inequality (\ref{1}), then we get%
\begin{equation*}
\left\langle -\mu x^{\ast },x-x^{\ast }\right\rangle \leq 0,\text{ }\forall
x\in \tciFourier .
\end{equation*}%
So we have%
\begin{equation*}
\left\langle x^{\ast },x^{\ast }-x\right\rangle =\left\langle x^{\ast
},x^{\ast }\right\rangle -\left\langle x^{\ast },x\right\rangle \leq
0\Longrightarrow \left\Vert x^{\ast }\right\Vert ^{2}\leq \left\Vert x^{\ast
}\right\Vert \left\Vert x\right\Vert .
\end{equation*}%
Hence, $x^{\ast }$ is the unique solution to the quadratic minimization
problem $x^{\ast }=\func{argmin}_{x\in \tciFourier }\left\Vert x\right\Vert
^{2}$. This completes the proof.
\end{proof}

From Theorem \ref{X}, we can deduce the following interesting corollaries.

\begin{corollary}
\label{Y1} Let $C$ be a nonempty closed convex subset of a real Hilbert
space $H.$ Let $S:C\rightarrow H$ be a nonexpansive mapping and $\left\{
T_{n}\right\} $ be a sequence of nonexpansive mappings such that $%
\tciFourier \neq \emptyset $. Suppose that $Tx=\lim_{n\rightarrow \infty
}T_{n}x$ for all $x\in C$. Let $V:C\rightarrow H$ be a $\gamma $%
-Lipschitzian mapping, $F:C\rightarrow H$ be a $L$-Lipschitzian and $\eta $%
-strongly monotone operator such that these coefficients satisfy $0<\mu <%
\frac{2\eta }{L^{2}}$, $0\leq \rho \gamma <\nu $, where $\nu =1-\sqrt{1-\mu
\left( 2\eta -\mu L^{2}\right) }$. For an arbitrarily initial value $%
x_{1}\in C,$ consider the sequence $\left\{ x_{n}\right\} $ in $C$ generated
by (\ref{4}) where $\left\{ \alpha _{n}\right\} $ and $\left\{ \beta
_{n}\right\} $ are sequences in $\left[ 0,1\right] $ satisfying the
conditions (C1)-(C3) of Theorem \ref{X} except the condition $%
\lim_{n\rightarrow \infty }\frac{a_{n}}{\alpha _{n}}=0$. Then, the sequence $%
\left\{ x_{n}\right\} $ converges strongly to $x^{\ast }\in \tciFourier $,
where $x^{\ast }$\ is the unique solution of variational inequality (\ref{1}%
).
\end{corollary}

Let $\lambda _{i}>0$ ($i=1,2,3,\ldots N$) such that $\tsum_{i=1}^{N}\lambda
_{i}=1$ and $T_{1},T_{2},\ldots T_{N}$ be nonexpansive self mappings on $C$
such that $\tbigcap_{i=1}^{N}Fix\left( T_{i}\right) \neq \emptyset $. Then, $%
\tsum_{i=1}^{N}\lambda _{i}T_{i}$ is nonexpansive self mapping on $C$ (see 
\cite[Proposition 6.1]{WOSY}).

\begin{corollary}
\label{Z} Let $C$ be a nonempty closed convex subset of a real Hilbert space 
$H.$ Let $\lambda _{i}>0$ ($i=1,2,3,\ldots N$) such that $%
\tsum_{i=1}^{N}\lambda _{i}=1$ and $S,T_{1},T_{2},\ldots T_{N}$ be
nonexpansive self mappings on $C$ such that $\tbigcap_{i=1}^{N}Fix\left(
T_{i}\right) \neq \emptyset $. Let $V:C\rightarrow H$ be a $\gamma $%
-Lipschitzian mapping, $F:C\rightarrow H$ be a $L$-Lipschitzian and $\eta $%
-strongly monotone operator such that these coefficients satisfy $0<\mu <%
\frac{2\eta }{L^{2}}$, $0\leq \rho \gamma <\nu $, where $\nu =1-\sqrt{1-\mu
\left( 2\eta -\mu L^{2}\right) }$. For an arbitrarily initial value $%
x_{1}\in C,$ consider the sequence $\left\{ x_{n}\right\} $ in $C$ generated
by%
\begin{equation}
\left\{ 
\begin{array}{c}
y_{n}=\beta _{n}Sx_{n}+\left( 1-\beta _{n}\right) x_{n}\text{, \ \ \ \ \ \ \
\ \ \ \ \ \ \ \ \ \ \ \ \ \ \ \ \ \ \ \ \ \ \ \ \ \ \ } \\ 
x_{n+1}=P_{C}\left[ \alpha _{n}\rho Vx_{n}+\left( I-\alpha _{n}\mu F\right)
\tsum_{i=1}^{N}\lambda _{i}T_{i}y_{n}\right] \text{, }\forall n\geq 1%
\end{array}%
\right.  \label{5}
\end{equation}%
where $\left\{ \alpha _{n}\right\} $ and $\left\{ \beta _{n}\right\} $ are
sequences in $\left[ 0,1\right] $ satisfying the conditions (C1) and (C2) of
Theorem \ref{X} except the condition $\lim_{n\rightarrow \infty }\frac{a_{n}%
}{\alpha _{n}}=0$. Then, the sequence $\left\{ x_{n}\right\} $ in $C$
generated by (\ref{5}) converges strongly to $x^{\ast }\in
\tbigcap_{i=1}^{N}Fix\left( T_{i}\right) $, where $x^{\ast }$\ is the unique
solution of variational inequality%
\begin{equation*}
\left\langle \left( \rho V-\mu F\right) x^{\ast },x-x^{\ast }\right\rangle
\leq 0,\text{ }\forall x\in \tbigcap_{i=1}^{N}Fix\left( T_{i}\right) .
\end{equation*}
\end{corollary}

\begin{remark}
Our results can be reduced to some corresponding results in the following
ways:

\begin{enumerate}
\item In our iterative process (\ref{4}), if we take $S=I$ ($I$ is the
identity operator of $C$), then we derive the iterative process (\ref{f3})
which is studied by Sahu et. al. \cite{SKS}. Therefore, Theorem \ref{X}
generalizes the main result of Sahu et. al. \cite[Theorem 3.1]{SKS}.\ Also,
Corollary \ref{Y1} and Corollary \ref{Z} extends the Corollary 3.4 and
Theorem 4.1 of Sahu et. al. \cite{SKS}, respectively. So, our results
extends the corresponding results of Ceng et. al. \cite{Ceng} and of many
other authors.

\item If we take $S$ as a nonexpansive self mapping on $C$ and\ $T_{n}=T$
for all $n\geq 1$ such that $T$ is a nonexpansive mapping in (\ref{4}), then
we get the iterative process (\ref{f4}) of Wang and Xu. \cite{WX}. Hence,
Theorem \ref{X} generalizes the main result of Wang and Xu \cite[Theorem 3.1]%
{WX}. So, our results extend and improve the corresponding results of \cite%
{T1, MX1}.

\item The problem of finding the solution of variational inequality (\ref{1}%
), is equivalent to finding the solutions of hierarchical fixed point
problem 
\begin{equation*}
\left\langle \left( I-S\right) x^{\ast },x^{\ast }-x\right\rangle \leq
0,\forall x\in \tciFourier ,
\end{equation*}%
where $S=$ $I-\left( \rho V-\mu F\right) .$
\end{enumerate}
\end{remark}

\end{document}